\documentclass[reqno,12pt]{amsart}
\usepackage[dvips]{color}
\usepackage[hypertex]{hyperref}

\usepackage{amsfonts}
\usepackage{amsthm}
\usepackage{amssymb}

\usepackage[english]{babel}
\usepackage[utf8]{inputenc}

\newtheorem{teo}{Theorem}
\newtheorem{prop}{Proposition}
\newtheorem{ex}{Example}
\newtheorem{cor}{Corollary}

\newtheorem{lem}{Lemma}

\newtheorem*{teorm}{Theorem}

\theoremstyle{definition}
\newtheorem{obs}{Remark}
\newtheorem{ddef}{Definition}

\newcommand{\co}{\mathbb{C}}

\newcommand{\cpt}[1]{\mathbb{C}P^{2}}

\newcommand{\pcn}[1]{\mathbb{P}^{#1}_{ \mathbb{C}}}

\newcommand{\pe}{\mathbb{P}}

\newcommand{\sing}{\mbox{Sing}}
\newcommand{\cl}[1]{\mathcal{#1}}
\newcommand{\ba}{\setminus}

\newcommand{\dr}{\mbox{$\partial$}}

\begin{document}

\setcounter{section}{0}
\setcounter{teo}{0}
\setcounter{exe}{0}

\author{Felipe Cano \& Nuria Corral \& Rog\'erio   Mol}
\title{Local polar invariants for plane singular foliations}
\dedicatory{To the memory of Marco Brunella}
\maketitle

\begin{abstract} In this survey paper, we take the viewpoint of polar invariants to the local and global study of
  non-dicritical holomorphic foliations in dimension two and their invariant curves.  It appears a characterization of second type foliations and generalized curve foliations as well as a description of the $GSV$-index in terms  of polar curves. We also interpret the proofs concerning the Poincar\'e problem with polar invariants.
\end{abstract}
\footnotetext[1]{ {\em 2000 Mathematics Subject Classification.}
Primary 32S65 ; Secondary 14C21. } \footnotetext[2]{{\em
Keywords.} Holomorphic foliation, polar curves, invariant curves.}
\footnotetext[3]{First and second authors supported by the Ministerio de Econom\'{\i}a y Competitividad MTM2013-46337-C2-1-P. Third author supported by CAPES, FAPEMIG, Pronex/FAPERJ and Universal/CNPq.}

 \medskip \medskip
\tableofcontents
\section{Introduction}

In this paper, we give a look based on the study of
intersection properties of polar curves of a foliation to three subjects concerning non-dicritical singular holomorphic foliations:
\begin{enumerate}
\item[a)] The characterization of second type foliations and generalized curves.
\item[b)] A polar interpretation of the GSV-index.
\item[c)] The (non-dicritical) Poincar\'{e} problem for foliations in ${\mathbb P}^2_{\mathbb C}$.
\end{enumerate}

The polar curves of a singular holomorphic foliation $\mathcal F$ in dimension two have been studied in a local way by P. Rouill\'e \cite{rouille1999} and N. Corral \cite{corral2003} and in a global way by R. Mol \cite{mol2010}. The definition is quite geometrical. Namely, the polar curve of $\mathcal F$ with respect to another foliation $\mathcal L$ is just the curve of tangencies between $\mathcal F$ and $\mathcal L$. It generalizes a classical approach of polar curves by E. Casas-Alvero, B. Teissier and others \cite{casas,Tei,Mer,Gar}. Usually $\mathcal L$ is a linear foliation of parallel lines (or lines passing through a point of ${\mathbb P}^2_{\mathbb C}$) and in this way the ``generic object'' is well defined with respect to certain criteria, such as the equisingularity, for instance.

Let us recall that the term
{\em generalized curve} comes from the results in \cite{camacho1984} and denotes a local foliation without saddle-nodes in its
 desingularization. Such foliations  have the outstanding properties of being desingularized by the same sequence of blow-ups that de\-sin\-gu\-la\-ri\-zes
the set of separatrices and minimizing Milnor number. The non-dicritical foliations whose reduction of singularities coincides with the reduction of singularities of the set of formal separatrices are those that minimize the
algebraic multiplicity.  This characterization is due to J.-F. Mattei and E. Salem \cite{mattei2004} and they use the terminology
{\em foliations of second type}.

The intersection number
of a generic polar curve of a non-dicritical foliation $\cl{F}$ and
a formal invariant curve $S$ at $0 \in \co^{2}$ is what we call {\em polar intersection number}  $p_0({\cl{F}},S)$. We have that
\begin{equation}
\label{eq:desigualdad}
p_0({\cl{F}},S)\leq \nu_0({\mathcal F})+\mu_0({\mathcal F}),
\end{equation}
where $\nu_0({\mathcal F})$ is the algebraic multiplicity and $\mu_0({\mathcal F})$ the Milnor number.
Moreover,  equality holds if and only if $\mathcal F$ is of second type and $S$ is its set of formal separatrices.

Take a germ of convergent invariant curve $C\subset S$  and a generalized curve $\mathcal G$ such that $C$ is its set of separatrices. Then we have
$$
p_0({\cl{G}},C) \leq p_0({\mathcal F},S)
$$
and equality holds if and only if $\mathcal F$ is a generalized curve and $S=C$ is its set of separatrices.

There is no obstruction to consider formal foliations and formal generalized curves. In this context, we also have that
for any formal generalized curve $\hat {\mathcal G}$ such that $S$ is its set of separatrices and any formal foliation $\hat {\mathcal F}$ such that  $S$ is invariant,   we have
$$
\Delta_0(\hat{\mathcal F},S)= p_0({\hat {\mathcal F}},S)-
p_0({\hat {\mathcal G}},S)_{0} \geq 0
$$
and  equality holds if and only if $\hat {\mathcal F}$ is a generalized curve and $S$ is its set of separatrices.

These results express a well known characterization of
generalized curves in terms of the G\'{o}mez-Mont-Seade-Verjovsky index. Take a convergent curve $C$ and a generalized curve $\mathcal G$ such that $C$ is its set of separatrices. Consider a non-dicritical foliation $\mathcal F$ such that $C$ is invariant by $\mathcal F$. We have that
$$
GSV_0({\mathcal F},C)= \Delta_0({\mathcal F},S)\geq 0.
$$
We recover in this way the known result that $\mathcal F$ is a generalized curve if and only $GSV_0(\cl{F},C) = 0$ (\cite{brunella1997II,lehmann2001}) and in this case $C$ is its set of separatrices. Moreover, the above formula gives a way of generalizing the GSV-index to formal invariant curves,  as well as another interpretation of the non-negativity of the GSV-index for non-dicritical foliations (see Proposition 6 of \cite{brunella1997II}).

Let $\cl{F}$ be a holomorphic foliation on $\pcn{2}$.  The degree of $\cl{F}$ is the number $\deg(\cl{F})$ of
tangencies  between $\cl{F}$ and a generic line
$L \subset \pcn{2}$.
The question concerning the existence of a bound for the degree of
an  algebraic curve $S$ invariant by  $\cl{F}$ in  terms of the
degree of $\cl{F}$ is known in Foliation Theory as the {\em
Poincar\'e problem}, being proposed by Poincar\'e himself in  \cite{poincare1891} as a step in
finding a rational first integral for a polynomial differential equation in two complex variables. A first answer to
this problem was given
 by D. Cerveau and A. Lins Neto \cite{cerveau1991}: if $S$
has at most nodal singularities, then  $\deg(S) \leq \deg(\cl{F}) +
2$. Besides,     this bound is reached if and only if
$\cl{F}$ is a logarithmic foliation, that is, a foliation induced by a closed meromorphic 1-form
with simple poles. Later, M. Carnicer obtained in
\cite{carnicer1994} the same inequality, removing the hypothesis on
the singularities of $S$, but  admitting that the singularities of
$\cl{F}$   over $S$  are all non-dicritical, meaning that the number of local separatrices
is finite.  The Poincar\'e problem was put in
a clearer setting by M. Brunella in the works \cite{brunella1997I}
and \cite{brunella1997II}.  Brunella has shown that the bound  $\deg(S) \leq
\deg(\cl{F}) + 2$ occurs whenever the sum over $S$ of the $GSV$-indices of $\cl{F}$ with respect to the local separatrices contained
in $S$ is non-negative. This happens in the two cases
mentioned above. In the dicritical case, N. Corral and P. Fern\'andez-S\'anchez \cite{Corral-Fdez2006}
proved that the degree of an algebraic invariant curve $S$ of $\mathcal F$ is bounded in terms of the degree of $\mathcal F$ provided that
the germ of $S$  at each singular point of $\mathcal F$ is a union of isolated separatrices of $\mathcal F$.

The last section is devoted to give a proof of Carnicer's bound in terms of local and global polar curves. We end the paper by a remark on logarithmic foliations, that corresponds to the limit case.

\section{Recall on local invariants and reduction of singularities}
We recall \cite{Cano-Cerveau} that a germ of singular holomorphic foliation $\cl{F}$ in $(\co^2,0)$ is  defined by $\omega=0$, where $\omega$ is a 1-form
 \begin{equation}
 \label{eq:uno}
 \omega = P(x,y)
dx + Q(x,y) dy
 \end{equation}or by the vector field ${\bf v} = - Q(x,y) \dr / \dr x + P(x,y) \dr / \dr y$, where $P,Q \in {\mathbb C}\{x,y\}$ are relatively prime. The origin $0$ is a {\em singular point} if $P(0,0)=Q(0,0)=0$. Note that any nonzero 1-form defines in a unique way a germ of singular holomorphic foliation, just by taking the common factor of the coefficients; same remark when we consider a vector field.

  A formal curve $C$ in $({\mathbb C}^2,0)$ is given by a reduced equation $f(x,y)=0$ with $f\in {\mathbb C}[[x,y]]$. When $f$ is irreducible, we say that $C$ is a {\em branch}. If
 $
 f=f_1f_2\cdots f_e
 $
 is the decomposition of $f$ as a product of irreducible factors, we say that each $B_i=(f_i=0)$ is one of the {\em branches of $C$}. If we can take a convergent equation $f\in {\mathbb C}\{x,y\}$, then $C$ is a germ of analytic curve, in this case, the branches are also analytic germs of curve.

 Let $\sigma: (M,\sigma^{-1}(0)) \to
(\co^2,0)$ be the blow-up of the origin.  Given a branch $B$ with $f(x,y)=0$, we know that $f$ has the form
$$
f(x,y)= (\lambda x+\mu y)^r+\tilde f(x,y), \ \ (\lambda,\mu)\ne (0,0),
$$
where all the terms of $\tilde f(x,y)$ have degree greater than $r$. We say that $r$ is the multiplicity of $B$ at the origin and we put $r=\nu_0(B)$. The line $\lambda x+\mu y=0$ is the {\em tangent cone} of $B$ at the origin and it determines a single point $\tau(B)$ in the projective line $\sigma^{-1}(0)$. Up to a linear coordinate change, we can assume that $f=y^r+\tilde f(x,y)$ and in this case, we have a branch $B'$ at $\tau(B)$ given by $f'(x',y')=0$, where
$$
x'^rf'(x',y')= f(x',x'y').
$$
 We say that $B'$ is the {\em strict transform of $B$}. By taking the union of branches, we define the strict transform of any formal curve $C$ and, by iterating the procedure, we can define the strict transform of $C$ under any sequence
 $$
 \pi: ({M},E)\rightarrow ({\mathbb C}^2,0)
 $$
 of punctual blow-ups. Note that each branch $B$ of $C$ gives a point $\tau_E(B)\in E$ in the exceptional divisor $E$.  We say that the sequence $\pi$ is a {\em reduction of singularities} of $C$ if and only if
 \begin{enumerate}
 \item For any two branches $B_1,B_2$ of $C$ we have $\tau_E(B_1)\ne \tau_E(B_2)$.
 \item For any branch $B$ of $C$, the exceptional divisor $E$ has only one irreducible component $D$ through $p=\tau_E(B)$ and the strict transform $B'$ of $B$ is non singular at $p$ and transversal to $E$. In other words, there are formal coordinates $(x,y)$ at $p$ such that $E=(x=0)$ and $B'=(y=0)$.
 \end{enumerate}
 It is known \cite{Zariski} that any formal curve has a reduction of singularities. Moreover, there is a minimal one, and any other factorizes through the minimal one by making additional punctual blow-ups. Let us note that doing a reduction of singularities needs at least one blow-up, even in the case that we have a non-singular branch.

The {\em transform} $\pi^{*} \cl{F}$ of a foliation $\cl{F}$ by $\pi$ is locally given by $\pi^*\omega$. Note that this definition is also valid for a {\em formal foliation}, that is, for a foliation given by $\hat\omega=0$, where the coefficients of $\hat \omega$ are formal series without common factor.

Take a singular holomorphic foliation $\mathcal G$ in $(M,E)$ and a point $p\in E$. We recall \cite{Cano-Cerveau} that $p$ is a {\em simple point} for ${\mathcal G},E$ if there are local coordinates $(x,y)$ centered at $p$ such that $E\subset (xy=0)$ locally at $p$ and one of the following properties holds
\begin{enumerate}
\item $\mathcal G$ is locally given by $dx=0$. This is the case when $p$ is non singular. If $(x=0)\subset E$, we say that $p$ is of {\em corner type} and otherwise $p$ is of {\em trace type}.
\item {\em Corner type singular points}: $\mathcal G$ is locally given by $\omega=xy\eta=0$ and the divisor is $E=(xy=0)$,  with
$$
\eta={a(x,y)}\frac{dx}{x}+b(x,y)\frac{dy}{y},
$$
where $(a(0,0),b(0,0))=(-\mu,\lambda)$, with $\mu\ne 0$ and $\lambda/\mu\notin {\mathbb Q}_{>0}$. We have two possibilities
\begin{enumerate}
\item $\lambda\ne 0$. This is a {\em complex hyperbolic singularity of corner type}, following the terminology of \cite{Can-R-S}.
\item $\lambda=0$.  This is a {\em badly oriented saddle-node of corner type}.
\end{enumerate}
\item {\em Trace type singular points}: $\mathcal G$ is locally given by $\omega=x\eta=0$ and the divisor is $E=(x=0)$,  with
$$
\eta={a(x,y)}\frac{dx}{x}+b(x,y){dy},
$$
where $a(x,y)=-\mu y +\alpha x+\tilde a(x,y)$, with $\tilde a(x,y)$ having all terms of degree at least two
and $b(0,0)=\lambda$. In addition, we ask that one of the following situations holds
\begin{enumerate}
\item $\lambda\mu\ne 0$ and $\lambda/\mu\notin {\mathbb Q}_{>0}$. This is a {\em complex hyperbolic trace singularity}.
\item $\mu\ne 0$ and $\lambda=0$. This is a {\em well oriented saddle node of trace type}.
\item $\mu = 0$ and $\lambda\ne 0$. This is a {\em badly oriented saddle node of trace type}.
\end{enumerate}
\end{enumerate}
{ In \cite{mattei2004},   badly oriented saddle nodes are called tangent saddle-nodes.} An irreducible component $D$ of $E$ may be invariant or {\em dicritical}. If we are dealing with a simple point $p\in M$, a dicritical component $D$ only exists when $p$ is non singular and in this case, up to choosing local coordinates $(x,y)$, we have that ${\mathcal G}$ is given by $dx=0$ and $D=(y=0)$.

As a direct consequence of
Seidenberg's
Desingularization Theorem \cite{Sei,Cano-Cerveau} there is a morphism  $\pi$ as above for which
 $\pi^*{\cl{F}}$ has only simple singularities. Such $\pi$ is called a {\em reduction of the singularities of} $\mathcal F$. Note that there is a minimal such $\pi$ and any other reduction of singularities of $\mathcal F$ factorizes through the minimal one by an additional sequence of blow-ups.

A {\em separatrix} for $\cl{F}$ in $(\co^2,0)$ is an invariant formal irreducible curve. Thus, it is given by an equation $f(x,y)=0$, where $f\in {\mathbb C}[[x,y]]$ is an irreducible formal series such that $f(0,0)=0$ and there is a formal series
$h\in {\mathbb C}[[x,y]]$ such that
$$
\omega\wedge df=(fh) dx\wedge dy.
$$
If we can take  $f\in {\mathbb C}\{x,y\}$,  the separatrix is convergent. We denote $\mbox{\rm Sep}({\mathcal F})$ the set of separatrices of $\mathcal F$.

Consider a reduction of singularities $\pi: (M,E) \to (\co^2,0)$ of $\mathcal F$. The following are equivalent
\begin{enumerate}
\item There is a dicritical component $D$ of $E$ for $\pi^*{\mathcal F}$.
\item The foliation $\mathcal F$ has infinitely many separatrices.
\end{enumerate}
Such foliations are called {\em dicritical}. In this paper we deal with non-dicritical foliations, that is foliations having only finitely many separatrices.  In this case, the union $S_{\mathcal F}=\bigcup\{B; B\in {\mbox{\rm Sep}}({\mathcal F})\}$ is a formal curve whose irreducible components are the separatrices.

The separatrices of a non-dicritical $\mathcal F$ are in one to one correspondence with the singular points of trace type in the reduction of singularities. To any separatrix $B$ we associate the singular trace point { $\tau_E(B)\in E$} through which it passes the strict transform of $B$.  The separatrix $B$ is called {\em strong} or of {\em Briot and Bouquet} type (see \cite{Lorena}, where this terminology is used) if and only if  $\tau_E (B)$ is either a complex hyperbolic singularity or a badly oriented saddle node. Such separatrices are convergent by application of the classical Briot and Bouquet Theorem. Note that in the classical Camacho-Sad paper \cite{CamachoSad} the authors show the existence of a Briot and Bouquet separatrix in order to prove that any non-dicritical foliation has at least one convergent separatrix.

\begin{ddef} Let $\mathcal F$ be a non-dicritical foliation in $({\mathbb C}^2,0)$ and consider the minimal desingularization $\pi: (M,E) \to (\co^2,0)$ of $\mathcal F$.
\begin{enumerate}
\item [(a)] The foliation $\mathcal F$ is a {\em generalized curve} (also {\em complex hyperbolic}) if all the singularities of $\pi^*{\cl{F}}$ are of complex hyperbolic type.
\item [(b)]  The foliation $\mathcal F$ is  of  {\em second type} if all the saddle nodes of $\pi^*{\cl{F}}$ are well oriented with respect to $E$.
\end{enumerate}
\end{ddef}
The terminology comes from previous papers \cite{camacho1984,mattei2004,Can-R-S}. Note that $\mathcal F$ is of second type if and only if  the strong separatrices  correspond to complex hyperbolic trace points and all the corners are also complex hyperbolic. The fact of being a generalized curve or  a second type foliation is independent of considering another, may be not minimal, reduction of singularities.

\subsection {Local invariants}
Let us recall now some of the local invariants frequently used in the local study of singular foliations in dimension two, see also \cite{Cano-Cerveau}.

The {\em algebraic multiplicity} $\nu_{0}(\cl{F})$ is the minimum of the orders $\nu_0(P)$, $\nu_0(Q)$ at the origin of the coefficients of a local generator of $\mathcal F$.  The {\em Milnor number} $\mu_0(\mathcal F)$ is given by
$$
\mu_0(\mathcal F)=\dim_{\mathbb C}\frac{{\mathbb C}[[x,y]]}{(P,Q)}=i_0(P,Q)
$$
(where $i_0(P,Q)$ stands for the intersection multiplicity).

Take a primitive parametrization $\gamma: ({\mathbb C},0)\rightarrow ({\mathbb C}^2,0)$, $\gamma(t)=(x(t),y(t))$, of a formal irreducible curve $B=(f(x,y)=0)$ at $({\mathbb C}^2,0)$. Then $B$ is a separatrix of $\mathcal F$ if and only if $\gamma^*\omega=0$. In this case, we can consider the {\em Milnor number $\mu_0({\mathcal F},B)$  of $\mathcal F$ along $B$} defined by
\[ \mu_0(\cl{F},B) = {\rm ord}_{t} {\bf w}(t) ,\]
where ${\bf w}(t)$ is the unique vector field at $(\co,0)$ such that
$\gamma_{*} {\bf w}(t) = {\bf v} \circ \gamma(t)$ (this number is also called multiplicity of $\mathcal F$ along $B$, see \cite{camacho1984,Cam-C-S}). We have that
\begin{equation}
\label{multiplicidaderelativa1}
 \mu_p(\cl{F},B) =
\begin{cases}
 {\rm ord}_{t}(Q(\gamma(t))) -  {\rm ord}_{t}(x(t)) + 1 \ \ \ \mbox{if $x(t) \neq 0$}   \\
 {\rm ord}_{t}(P(\gamma(t))) -  {\rm ord}_{t}(y(t)) + 1 \ \ \ \mbox{if $y(t) \neq 0$}
\end{cases}
\end{equation}
 If $B$ is not a separatrix, we define {\em the tangency order $\tau_0({\mathcal F},B)$} to be the ${\rm ord}_{t}a(t)$ where ${\gamma^* w}=a(t)dt$.

These invariants have a behavior under blow-up that helps in many of the results we are concerned here. For instance, if $B$ is not a separatrix and we consider the blow-up of the origin  $\sigma: (M,\sigma^{-1}(0))\rightarrow ({\mathbb C}^2,0)$, we have
\begin{equation}
\label{eq:tangenciyorder}
\tau_0({\mathcal F},B)= \nu_0(B)\nu_0({\mathcal F})+ \tau_{p'}(\sigma^*{\mathcal F},B'),
\end{equation}
where $B'$ is the strict transform of $B$ by $\sigma$ and $\{p'\}=B'\cap \sigma^{-1}(0)$.

\begin{ex}
\label{saddle}
{\rm
For a   complex hyperbolic singularity with two transversal separatrices $B_1$ and $B_2$, we have $\mu_0(\cl{F},B_1) = \mu_0(\cl{F},B_2) = 1$.
For a  saddle-node, up to reordering,  we have $\mu_0({\mathcal F,B_1})=1$, $\mu_0({\mathcal F,B_2})=k+1$ where $B_1$ is the {\em strong separatrix} and $k\geq 1$ is the {\em Poincar\'{e} order} of the saddle node. }
\end{ex}

\subsection{Comparison with the hamiltonian foliation}
Let $S$ be a formal curve and $f=0$ be a reduced equation for $S$. We denote ${\mathcal G}_f$ the ``hamiltonian'' foliation defined by $df=0$. Then ${\mathcal G}_f$ is non-dicritical and $S$ is its curve of separatrices.  Note that
$$
\nu_0({\mathcal G}_f)=\nu_0(S)-1
$$
and $\mu_0({\mathcal G}_f)$ corresponds exactly to the usual definition of the Milnor number of $S$ at the origin.
We will frequently compare invariants of a non-dicritical $\mathcal F$ and ${\mathcal G}_f$, where $f=0$ is a reduced equation of  $S_{\mathcal F}$.

Non-dicritical second type foliations minimize the algebraic multiplicity and have the same reduction of singularities as $S_{\mathcal F}$. Moreover, generalized curves also minimize the Milnor number. Let us state these results as Mattei-Salem in \cite{mattei2004}

 \begin{teo}[\cite{camacho1984, mattei2004}]
  \label{teo:matteisalem}
  Let $\mathcal F$ be a non-dicritical foliation and consider ${\mathcal G}_f$ where $f=0$ is a reduced equation of $S_{\mathcal F}$. Take the minimal reduction of singularities $\pi:(M,E)\rightarrow ({\mathbb C}^2,0)$ of $\mathcal F$. Then
 \begin{enumerate}
 \item $\pi$ is a reduction of singularities of $S_{\mathcal F}$. Moreover $\pi$ is the minimal reduction of singularities of $S_{\mathcal F}$ if and only if ${\mathcal F}$ is of second type.
 \item $\nu_0({\mathcal F})\geq \nu_0({\mathcal G}_f)$. Equality holds if and only if $\mathcal F$ is of  second type.
 \item $\mu_0({\mathcal F})\geq \mu_0({\mathcal G}_f)$. Equality holds if and only if $\mathcal F$ is a generalized curve.
 \end{enumerate}

 \end{teo}

Next corollary is a more general  version of a  result in
\cite{rouille1999}.

\begin{cor}
\label{cor:secondtypetangencyorder}
Consider a non-dicritical foliation ${\mathcal F}$ and take a branch $B$ which is
 not a separatrix.  Then
 $$
 i_0(S_{\mathcal F},B)\leq \tau_0 ({\mathcal F},B)+1
 $$
and equality holds if and only if ${\mathcal F}$ is of second type.
\end{cor}
\begin{proof} We make induction on the number $n$ of blow-ups needed to obtain the situation that $\mathcal F$ is non singular at the point $p\in B$ we are considering, the branch $B$ is non singular and transversal to $\mathcal F$. Note that this number $n$ exists since $B$ is not a separatrix. If $n=0$ we are done since we can assume  that $\mathcal F$ is given by $dx=0$ and $B=(y=0)$. Let $\sigma: (M,\sigma^{-1}(0))\rightarrow ({\mathbb C}^2,0)$ be the blow-up of the origin, where the exceptional divisor is the projective line  $D=\sigma^{-1}(0)$. Denote by $S_{\mathcal F}',B'$ the respective strict transforms of $S_{\mathcal F},B$ and let $p'$ be the point $D\cap B'$. By Noether's formula, we have
\begin{equation}
i_0(S_{\mathcal F},B)=\nu_0(S_{\mathcal F})\nu_0(B)+i_{p'}(S_{\mathcal F}',B').
\end{equation}
In view of Theorem \ref{teo:matteisalem} we have $\nu_0(S_{\mathcal F})\leq \nu_0({\mathcal F})+1$ and  equality holds if and only $\mathcal F$ is of second type. By the induction hypothesis and recalling that the separatrices of $\sigma^*{\mathcal F}$ at $p'$ are given by $S_{\mathcal F}'\cup D$, we have that
$$
i_{p'}(S_{\mathcal F}',B')+\nu_{0}(B)=
i_{p'}(S_{\mathcal F}',B')+i_{p'}(D,B')=i_{p'}(S_{\mathcal F}'\cup D,B')\leq  \tau_{p'} (\sigma^*{\mathcal F},B')+1
$$
where equality holds if and only if $\sigma^*{\mathcal F}$ is of second type at $p'$. Looking at Equation~\eqref{eq:tangenciyorder} we conclude that
\begin{eqnarray*}
i_0(S_{\mathcal F},B)&=&\nu_0(S_{\mathcal F})\nu_0(B)+i_{p'}(S_{\mathcal F}',B')\leq \\
&\leq& (\nu_0({\mathcal F})+1)\nu_0(B)+ \tau_{p'} (\sigma^*{\mathcal F},B')+1 - \nu_0(B)=\\
&=& \tau_{0} ({\mathcal F},B)+1
\end{eqnarray*}
and equality holds if and only if $\mathcal F$ (and hence $\sigma^*{\mathcal F}$) is of second type.
\end{proof}


\section{Polar intersection numbers}
Let $\cl{{F}}$ be a germ of singular foliation in $(\co^2,0)$  given by $\omega = P dx + Q dy$, where $P,Q$ are without common factors. The {\em polar
curve $P_{(a:b)}^\cl{F}$} of $\cl{F}$ with respect to $(a:b) \in \pcn{1}$ is
defined by the equation $a P + b Q = 0$.
In terms of differential forms, it is given by
$
\omega\wedge (bdx-ady)=0
$.
Note that $P_{(a:b)}^\cl{F}$ has no invariant branches unless $ax+by=0$ is an invariant line.
The definition of polar curve also makes sense for formal foliations and
the invariants to be defined below can be extended to the formal world.

Let us fix a formal curve $C$ invariant by $\mathcal F$. There is a non-empty Zariski open set $U_C\subset \pcn{1}$ such that for any $(a:b) \in U_C$ the polar $P_{(a:b)}^\cl{F}$ has no common branches with $C$ and the equisingularity type of $P_{(a:b)}^\cl{F}\cup C$ is  independent of $(a:b)\in U_C$. A  formal curve  $\Gamma$ in $({\mathbb C}^2,0)$   is of {\em $C$-generic polar type} iff $\Gamma\cup C$ is equisingular to $P_{(a:b)}^\cl{F}\cup C$ for $(a:b)\in U_C$ (see \cite{corral2003}). The number
$$
p_0({\mathcal F},C)=i_0(\Gamma,C)
$$
is independent of $\Gamma$ and we call it the {\em $C$-polar intersection number}.

\begin{obs}
Note that, since $(a:b)$ runs in a non empty Zariski open set, we have
 $$
 \nu_0(\Gamma) = \nu_0(\cl{F})=\min\{\nu_0(P),\nu_0(Q)\},
 $$
for any $\Gamma$ of $C$-generic polar type. Moreover, taking $\Gamma=P^{\mathcal F}_{(a:b)}$ for $(a:b)$ generic enough we have
$
p_0(\cl{F},C)= \min \{i_0(C,P),i_0(C,Q)\}
$.
\end{obs}
 Note that if $C\subset C'$ are invariant formal curves, any formal curve $\Gamma$ of $C'$-generic polar type is also of $C$-generic polar type. In the case that $\mathcal F$ is non-dicritical, any invariant formal curve $C$ is contained in the curve $S_{\mathcal F}$ of separatrices. In this situation, we say that $\Gamma$ is of {\em generic polar type} if it is of $S_{\mathcal F}$-generic polar type.

In this section we give two results. The first one concerns the polar intersection number with respect to a single separatrix and the second one relatively to the biggest invariant curve $S_{\mathcal F}$. In  next section we consider intermediate invariant curves.

\begin{prop}
\label{prop3} Consider a separatrix $B$ of a non-dicritical foliation $\mathcal F$. We have
\[ p_0(\cl{F},B) = \mu_0(\cl{F},B) +  \nu_0(B) - 1 .\]
\end{prop}
\begin{proof} Let  $\gamma(t)=(x(t),y(t))$ be a Puiseux parametrization for $B$ and assume without loss of generality that $x(t)\ne 0$ and hence $\dot x(t)\ne 0$. Taking a generic polar $aP+bQ=0$ we know that
$
p_0({\mathcal F},B)={\rm ord}_{t} (aP(\gamma(t))+bQ(\gamma(t)))
$.
Since $B$ is a separatrix, we have $P(\gamma(t)) \dot{x}(t)=- Q(\gamma(t)) \dot{y}(t)$ and applying Equation \eqref{multiplicidaderelativa1} we obtain
\begin{align}
p_0(\cl{F},B) & =   \displaystyle {\rm ord}_{t} \left( -a \frac{Q(\gamma(t)) \dot{y}(t)}{\dot{x}(t)} + b Q(\gamma(t)) \right)
     \nonumber \medskip \\
    & =    {\rm ord}_{t}Q(\gamma(t)) - ({\rm ord}_{t}x(t) - 1)
    + {\rm ord}_{t}(-a \dot{y}(t) + b \dot{x}(t))
    \nonumber \medskip \\
    & =     \mu_0(\cl{F},B)
    + {\rm ord}_{t}(-a y(t) + b x(t)) - 1
     \nonumber \medskip \\
    & =    \mu_0(\cl{F},B) + \nu_0(B) - 1. \nonumber
\end{align}
\end{proof}
Next result follows applying Corollary \ref{cor:secondtypetangencyorder} to the proof of Proposition 3.7 in
\cite{corral2003}. We include the proof for the sake of completeness.

\begin{prop} \label{prop2}
Let $\cl{F}$  be a non-dicritical foliation. Then
$$p_0(\cl{F},S_{\mathcal F}) \leq \mu_0(\cl{F}) + \nu_0(\cl{F})$$
and  equality holds if and only if $\cl{F}$ is
of second type.
\end{prop}
\begin{proof} Let $\Gamma=P^{\mathcal F}_{(a:b)}$ be a generic polar, with $a=1$ and $b$ generic enough. Denote by
${\mathcal B}(\Gamma)$ the set of irreducible components of $\Gamma$.
By Corollary \ref{cor:secondtypetangencyorder}, we know that
$$p_0({\mathcal F},S_{\mathcal F})=
\sum_{B\in {\mathcal B}(\Gamma)}i_0(B,S_{\mathcal F})\leq \sum_{B\in {\mathcal B}(\Gamma)}(\tau_0({\mathcal F},B)+1)
$$
and  equality holds if and only if $\mathcal F$ is of second type. Let us show that the last term is equal to $\mu_0(\cl{F}) + \nu_0(\cl{F})$. Choose a primitive parametrization $\gamma_B(t)=(x_B(t),y_B(t))$ for each $B\in {\mathcal B}(\Gamma)$.  If  $\omega=Pdx+Qdy$ defines the foliation, we recall that $P(\gamma_B(t))=-bQ(\gamma_B(t))$, since $B$ is a branch of $\Gamma$. Now
\begin{eqnarray*}
\sum_{B\in {\mathcal B}(\Gamma)}(\tau_0({\mathcal F},B)+1)&=& \sum_{B\in {\mathcal B}(\Gamma)}(
\text{ord}_t\{ P(\gamma_B(t))\dot x_B(t)+Q(\gamma_B(t))\dot y_B(t)\}+1)\\
&=& \sum_{B\in {\mathcal B}(\Gamma)}(
\text{ord}_t \{Q(\gamma_B(t)\}+ \text{ord}_t\{-b\dot x_B(t)+\dot y_B(t)\}+1)\\
&=& \sum_{B\in {\mathcal B}(\Gamma)}(
\text{ord}_t \{Q(\gamma_B(t)\}+ \text{ord}_t\{-b x_B(t)+ y_B(t)\})\\
&=& i_0(Q,\Gamma=P+bQ)+\nu_0(\Gamma)=\mu_0({\mathcal F})+\nu_0({\mathcal F}).
\end{eqnarray*}
Note that $b$ is generic. This ends the proof.
\end{proof}

\begin{obs}
\label{remark1-prop2} Let $C$ be a formal curve in $({\mathbb C}^2,0)$ and $f=0$ be a reduced equation of $C$. Let us consider the hamiltonian foliation ${\mathcal G}_f$ given by $df=0$. We know that ${\mathcal G}_f$ is a generalized curve and $C$ its curve of separatrices. Moreover, by definition of  Milnor number and multiplicity, we have that  $ \mu_0({\mathcal G}_f)=\mu_0(C)$ and $\nu_0({\mathcal G}_f)=\nu_0(C)-1$.
Then Proposition \ref{prop2} gives that
\begin{equation}
\label{eq:polarnumbercurvageneralizada}
p_0({\mathcal G}_f,C) = \mu_0(C)+ \nu_0(C)-1 ,
\end{equation}
and $p_0({\mathcal G}_f,C)$ does not depend on the choice of the reduced equation $f$. More generally, if ${\mathcal G}$ is a generalized curve such that $C=S_{\mathcal G}$ we also have that $p_0({\mathcal G},C) = \mu_0(C)+ \nu_0(C)-1$.
\end{obs}

\begin{obs}
\label{remark-prop2}
Following Theorem \ref{teo:matteisalem} and Proposition \ref{prop2}, we have
\begin{equation}
\label{eq0-remark-prop2}
p_0({\mathcal F},S_{\mathcal F}) = \mu_0(\cl{F})+ \nu_0(S_{\mathcal F})-1
\end{equation}
for a non-dicritical  foliation of second type $\mathcal F$. Taking Equation (\ref{eq0-remark-prop2}) for ${\mathcal G}_f$, where $f=0$ is a reduced equation of $S_{\mathcal F}$, we obtain that
\begin{equation}
\label{eq-remark-prop2}
p_0({\mathcal F},S_{\mathcal F})
- p_0({{\mathcal G}_f},S_{\mathcal F}) = \mu_0({\mathcal F})- \mu_0({\mathcal G}_f)= \mu_0({\mathcal F})- \mu_0({S}_{\mathcal F})\geq 0.
\end{equation}
 Note the positivity of this difference. In particular, a non-dicritical foliation  $\mathcal F$ of second type is a generalized curve if and only if $p_0({\mathcal F},S_{\mathcal F})
=p_0({{\mathcal G}_f},S_{\mathcal F})$.
\end{obs}

\begin{cor}
\label{remark-prop3}
For a non-dicritical foliation $\mathcal F$ of second type, we have
\[ \mu_0(\cl{F}) =1- \delta_{\mathcal F}+ \sum_{B \in  \text{\rm Sep}({\mathcal F})} \mu_0(\cl{F},B),\]
where $\delta_{\mathcal F}$ is the number of separatrices of $\mathcal F$.
\end{cor}
\begin{proof} By summing up polar intersection numbers over all separatrices we get
\[p_0(\cl{F},S_{\mathcal F}) =\sum_{B \in  \text{\rm Sep}({\mathcal F})} p_0(\cl{F},B) = \sum_{B \in  \text{\rm Sep}({\mathcal F})} \mu_0(\cl{F},B) +  \nu_0(S_{\mathcal F}) - \delta_{\mathcal F} ,\]
Since $\cl{F}$ is of second type we have $\nu_0(S_{\mathcal F})=\nu_0({\mathcal F})+1= p_0(\cl{F},S_{\mathcal F})- \mu_0({\mathcal F})+1
$ and we are done.
\end{proof}
\section{The formal GSV-index and polar intersection numbers}
In this section we give a formal extension of the GSV-index introduced by  X. G\'omez-Mont, J. Seade and
A. Verjovsky in \cite{gomezmont1991}.
\begin{ddef}\label{def:exceso}
Take a germ of singular foliation $\mathcal F$ and a formal curve $C$ invariant by $\mathcal F$. We define the {\em $C$-polar excess} $\Delta_0({\mathcal F},C)$ by
$$
\Delta_0({\mathcal F},C)=
p_0({\mathcal F}, C)- p_0({\mathcal G}, C)= p_0({\mathcal F}, C)-\mu_0(C)-\nu_0(C)+1 ,
$$
where $\mathcal G$ is any generalized curve such that $C=S_{\mathcal G}$.
\end{ddef}

Take a non-dicritical foliation $\mathcal F$ and a germ of (convergent) curve $C$ invariant by $\mathcal F$. Following  Brunella \cite{brunella1997II} we know that
$
GSV({\mathcal F}, C)\geq 0
$
and it is equal to $0$ in the case of a generalized curve. This positivity is a key argument in bounding degrees for Poincar\'{e} Problem in \cite{carnicer1994}. In this section, we show that
$$
GSV({\mathcal F}, C)=\Delta_0({\mathcal F}, C).
$$
Thus, the polar excess gives a formal extension of the GSV-index.

We start the section by proving that $\Delta_0({\mathcal F},C)\geq 0$ for any non-dicritical $\mathcal F$.

\subsection{Positivity of the polar excess} Before proving the positivity of $\Delta_0({\mathcal F},C)$ we consider the behavior of  polar intersection numbers under blow-up.

Consider the  blow-up
 $\pi: (M,E) \to (\co^2,0)$ of the origin of $({\mathbb C}^2,0)$. Note that if ${\mathcal F}$ is a non-dicritical foliation, then the exceptional divisor $E=\pi^{-1}(0)$ is an invariant projective line for $\pi^*{\mathcal F}$.
\begin{prop}
\label{blowup}
Let $B\in \text{Sep}({\mathcal F})$ be a separatrix of a non-dicritical foliation $\cl{F}$  in $(\co^2,0)$ and let $\widetilde{p}\in E$ be the only point in $\widetilde{B}\cap E$, where $\widetilde{B}$ is the strict transform of $B$. We have
\begin{equation}
\label{blowupeqn}
 p_{\widetilde{p}}(\pi^*{\cl{F}},\widetilde{B}) =
 p_0(\cl{F},B) + \nu_{\widetilde{p}}(\widetilde{B})
- \nu_0({\mathcal F})\nu_0(B).
\end{equation}
\end{prop}
\begin{proof} Up to a linear change of coordinates, we suppose that the tangent cone of $B$ is $y=0$ and thus we have local coordinates $(x,v)$ around $\widetilde{p}$ given by $v=y/x$. Take
 a Puiseux parametrization  ${\gamma}(t) = (t^n,\phi(t))$ of
${B}$, where $n=\nu_0(B)$ and $\text{ord}_t(\phi(t))>n$. Then, a primitive parametrization of $\widetilde{B}$ is given by $\widetilde{\gamma}(t)=(t^n,\widetilde{\phi}(t)=\phi(t)/t^n)$. If $\mathcal F$ is given by $\omega=Pdx+Qdy=0$, then $\pi^*{\mathcal F}$ is locally given at $\widetilde{p}$ by
$$
\omega'=x^{-\nu_0({\mathcal F})}\{(P(x,xv)+vQ(x,xv))dx+xQ(x,xv)dv\}.
$$
We have
$$
p_{\widetilde{p}}(\pi^*{\mathcal F}, \widetilde{B})=\text{ord}_t\{P(\gamma(t))+\widetilde{\phi}(t)Q(\gamma(t))+bt^nQ(\gamma(t))\}-n\nu_0({\mathcal F}),
$$
where $b$ is generic. Since $B$ is invariant, we have that $nt^{n-1}P(\gamma(t))=-\phi'(t)Q(\gamma(t))$ and thus
$$
p_{\widetilde{p}}(\pi^*{\mathcal F}, \widetilde{B})= \text{ord}_t\{Q(\gamma(t))\}+\text{ord}_t\left\{\frac{-t\phi'(t)+n\phi(t)}{nt^{n}}+bt^n \right\}
-n\nu_0({\mathcal F}).$$
For $b$ generic, we have
$$
\text{ord}_t\left\{\frac{-t\phi'(t)+n\phi(t)}{nt^{n}}+bt^n \right\}=\nu_{\widetilde{p}} (\widetilde{B}).
$$
Moreover, noting that $nt^{n-1}P(\gamma(t))=-\phi'(t)Q(\gamma(t))$ and $\text{ord}_t(\phi(t))>n$ we have
$$
\text{ord}_t({P(\gamma(t))})>\text{ord}_t({Q(\gamma(t))})
$$
and hence $\text{ord}_t({Q(\gamma(t))})= p_0({\mathcal F},B)$. This ends the proof.
\end{proof}

{
\begin{obs}
If we consider a polar curve $P^{\mathcal F}_{(a:b)}$ for a dicritical foliation $\mathcal F$ and we define  $p_0(\cl{F},B)=i_0(P^{\mathcal F}_{(a:b)},B)$ for generic $(a:b)$,  we obtain
\begin{equation*}
 p_{\widetilde{p}}(\pi^*{\cl{F}},\widetilde{B}) =
 p_0(\cl{F},B) + \nu_{\widetilde{p}}(\widetilde{B})
- (\nu_0({\mathcal F})+1)\nu_0(B),
\end{equation*}
when the blow-up of the origin is dicritical for the foliation $\mathcal F$.
\end{obs}
}
\begin{obs}
\label{remarkblowup}
 {\rm Take a curve $\Gamma$ of generic polar type and consider the situation of Proposition \ref{blowup}. Denote by $\tilde\Gamma$ the strict transform of $\Gamma$ by $\pi$. We recall that
 $$
 p_0({\mathcal F},B)=i_0(\Gamma,B).
 $$
 By Noether's formula, we also have
 $$
 i_{\widetilde{p}}(\widetilde{\Gamma},\widetilde{B})=i_{0}(\Gamma,B)-\nu_0(\Gamma)\nu_0(B)= p_0({\mathcal F},B)-\nu_0({\mathcal F})\nu_0(B).
 $$
 We obtain that
 $
 p_{\widetilde{p}}(\pi^*{{\mathcal F}},\widetilde{B})-i_{\widetilde{p}}(\widetilde{\Gamma},\widetilde{B})=\nu_{\widetilde{p}}(\widetilde{B})$. In particular, the strict transform of a curve of generic polar type is not a curve of generic polar type for the transformed foliation.
 }\end{obs}

\begin{teo}
\label{inequality}
Consider a non-dicritical foliation $\cl{F}$ in $(\co^2,0)$ and an invariant curve
$C\subset S_{\mathcal F}$. For any
  branch $B\in {\mathcal B}(C)$, we have
\[ p_0({\cl{G}_f},B) \leq p_0({\cl{F}},B) ,\]
where $f=0$ is a reduced
equation of $C$. Moreover, the following statements are equivalent
\begin{enumerate}
\item There is $B\in {\mathcal B}(C)$ such that $p_0({\cl{G}_f},B)= p_0({\cl{F}},B)$.
\item The foliation $\mathcal F$ is of second type with $C=S_{\mathcal F}$.
\end{enumerate}
Finally, if $\mathcal F$ is of second type with $C=S_{\mathcal F}$, a separatrix $B\in {\mathcal B}(S_{\mathcal F})$ is of Briot and Bouquet type if and only if $p_0({\cl{G}_f},B)= p_0({\cl{F}},B)$.
\end{teo}
Note that in Lemma 2.2. of  \cite{carnicer1994} it is proved that $\mu_0({\mathcal F},B) \geq \mu_0({\mathcal G},B)$ and hence the first statement is a consequence of Proposition \ref{prop3}. Anyway we give a complete proof below.

\begin{proof} Consider a sequence of local blow-ups $$\pi_k:(M_k,q_k)\rightarrow (M_{k-1},q_{k-1}), \quad k=1,2,\ldots,N$$ described as follows. First $(M_0,q_0)=({\mathbb C}^2,0)$ and $\pi_1$ is the blow-up of the origin. Next $\pi_k:(M_k,q_k)\rightarrow (M_{k-1},q_{k-1})$ is the blow-up centered at $q_{k-1}$, where $q_{k-1}$ is the only point in the strict transform $B_{k-1}\subset M_{k-1}$ of $B$, followed by the localization at the new ``infinitesimal near point'' $q_k$ of $B_k$. We denote by $E_k\subset M_k$ the total exceptional divisor, starting with $E_0=\emptyset$. Denote by $C_k\subset M_k$ the strict transform of $C$ in $M_k$ (localized around $q_k$) and by ${\mathcal F}_k$, respectively ${\mathcal G}_k$, the transforms of ${\mathcal F}$, respectively ${\mathcal G}_f$, at $M_k$. At the final step $N\geq 1$ we ask that $C$ is desingularized at $q_N$, that is $B_N=C_N$ locally at $q_N$ and $B_N$ has normal crossings with $E_N$ at $q_N$. In particular, there are local coordinates $(x,y)$ at $q_N$ such that $E_N=(x=0)$ and $B_N=(y=0)$. The existence of this sequence is an immediate consequence of the reduction of singularities of plane  curves.

Let us denote $\Delta_k=p_{q_k}({\mathcal F}_k, B_k)-p_{q_k}({\mathcal G}_k,B_k)$. Note that ${\mathcal G}_N$ is a generalized curve at $q_N$ with separatrices $xy=0$, that is, up to a coordinate change, it is given by a linear form $ydx+rxdy$ with $r\in {\mathbb Q}_{>0}$. In particular, we have $p_{q_N}({\mathcal G}_N,B_N)=1$. Moreover, $q_N$ is  a singular point for ${\mathcal F}_N$, since $xy=0$ are also invariant curves for ${\mathcal F}_N$. Hence $p_{q_N}({\mathcal F}_N,B_N)\geq 1$ and thus $\Delta_N\geq 0$. Let us show that
\begin{equation}
\label{eq:desigualdades}
\Delta_0\geq\Delta_1\geq \Delta_2\geq\cdots\geq \Delta_N\geq 0.
\end{equation}
Invoking Proposition \ref{blowup}, we have
\begin{equation}
\label{eq:deltas}
\Delta_{k-1}=\Delta_k+m_{k-1}(\nu_{q_{k-1}}({\mathcal F}_{k-1})- \nu_{q_{k-1}}({\mathcal G}_{k-1})),
\end{equation}
where $m_{k-1}=\nu_{q_{k-1}}(B_{k-1})\geq 1$. Since ${\mathcal G}_{k-1}$ is a generalized curve at $q_{k-1}$ whose separatrices are the irreducible components of $C_{k-1}\cup E_{k-1}$  and  $C_{k-1}\cup E_{k-1}\subset S_{{\mathcal F}_{k-1}}$, we have that
$$
\nu_{q_{k-1}}({\mathcal F}_{k-1})- \nu_{q_{k-1}}({\mathcal G}_{k-1})\geq 0.
$$
This already shows the statement of Equation \eqref{eq:desigualdades} and in particular $\Delta_0\geq 0$.

Now, assume $\Delta_0=0$. This implies that $\Delta_k=0$ for all $k=0,1,\ldots, N$.  In particular, by Equation \eqref{eq:deltas} applied to $k=1$ we deduce that $\nu_0({\mathcal F})=\nu_0({\mathcal G}_f)$, since $C\subset S_{\mathcal F}$, by Theorem \ref{teo:matteisalem}, this is only possible if $C=S_{\mathcal F}$ and $\mathcal F$ is of second type.

Conversely, assume that ${\mathcal F}$ is of second type with $C=S_{\mathcal F}$. Take a branch $B\in {\mathcal B}(S_{\mathcal F})$ and let us consider the local situation at $q_N$ where $B_N=(y=0)$ and $E_N=(x=0)$. Write a generator of ${\mathcal F}_N$ as
$
\omega_N= g(x,y)ydx+h(x,y)xdy
$. Then $p_{q_N}({\mathcal F}_N,B_N)\geq 1$ and $p_{q_N}({\mathcal F}_N,B_N)=1$ if and only if $h(0,0)\ne 0$, this is equivalent to say that $B$ is a Briot and Bouquet separatrix. That is we have shown that $B$ is a Briot and Bouquet separatrix if and only if $\Delta_N=0$. Moreover, at an intermediate point $q_k$ we have that both ${\mathcal F}_k$ and ${\mathcal G}_k$ are of second type with
$$
S_{{\mathcal F}_{k}}=S_{{\mathcal G}_k}=C_k\cup E_k.
$$
In particular $\nu_{q_{k-1}}({\mathcal F}_{k-1})- \nu_{q_{k-1}}({\mathcal G}_{k-1})=0$ and thus we have that
$$
\Delta_0=\Delta_1= \Delta_2=\cdots=\Delta_N\geq 0.
$$
Finally, note that by Camacho and Sad arguments in \cite{camacho1984} there is always a Briot and Bouquet separatrix. This ends the proof.
\end{proof}

\begin{cor}
\label{cor:positivity}
 We have $
 \Delta_0({\mathcal F},C)\geq 0
 $ for any
 non-dicritical foliation  $\cl{F}$ in $(\co^2,0)$ and any curve $C\subset S_{\mathcal F}$.
\end{cor}
\begin{proof}
It follows from the additivity of the polar intersection indices, by applying Theorem \ref{inequality} to all the branches $B\in {\mathcal B}(C)$.
\end{proof}

\begin{cor}
\label{cor:generalizada}
A non-dicritical  foliation  $\cl{F}$ in $(\co^2,0)$
is a generalized curve if and only if
$$
\Delta_0({\mathcal F},S_{\mathcal F})= 0.
$$  In this case we have
$
p_0({\mathcal F},S_{\mathcal F})=\mu_0(S_{\mathcal F})+\nu_0({S}_{\mathcal F})-1
$.
\end{cor}
\begin{proof} In view of Theorem \ref{inequality} and the additivity of the polar intersection numbers the stated equality is equivalent to say that
$$
p_0({\mathcal G}_f,B)= p_0({\mathcal F},B)
$$
for any separatrix $B\in \text{Sep}({\mathcal F})$. This is also equivalent to say that ${\mathcal F}$ is of second type and all the separatrices are of Briot and Bouquet, what is the same to say that $\mathcal F$ is a generalized curve. The second part is a consequence of Proposition \ref{prop2} (see also \cite{corral2003}).
\end{proof}

\subsection{Computation of GSV-index} Let us recall the definition in \cite{brunella1997II} of the GSV-index.
Let $C$ be a germ of (convergent) curve invariant by a foliation $\cl{F}$ in $(\co^2,0)$. Take a reduced equation $f=0$ of $C$ for $f\in {\mathbb C}\{x,y\}$ and a 1-form $\omega$ that defines $\mathcal F$. There is a decomposition
\begin{equation}
\label{decomp}
 g \omega = k df + f \eta ,\quad g,k\in {\mathbb C}\{x,y\},
\end{equation}
where $\eta$ is a holomorphic 1-form.

\begin{ddef}{\rm
 The $GSV$-{\em index} of $\cl{F}$ with
respect to $C$ at  the origin is defined by
\[ GSV_0(\cl{F},C) = \frac{1}{2 \pi i} \int_{\partial C} \frac{g}{k} d \left(\frac{k}{g} \right).\]
Here $\partial C = C \cap S^3_\epsilon$, where $S^3_\epsilon$ is a
small sphere centered at $0 \in \co^2$, oriented as the boundary
of $C \cap B^4_\epsilon$, for  a ball $B^4_\epsilon$ such that $S^3_\epsilon = \partial B^4_\epsilon$.
\label{GSVdefinition}
}
\end{ddef}

Next Lemma \ref{lem:unionexcesos} remakes the behavior of the CSV-index for the union of two sets of separatrices (see \cite{brunella1997II}, section 3).
\begin{lem}
 \label{lem:unionexcesos}
 Let $\mathcal F$ be a non-dicritical foliation and consider two curves $C_1, C_2\subset S_{\mathcal F}$ without common branches. Then
\[ \Delta_0(\cl{F},C_1 \cup C_2) =  \Delta_0(\cl{F},C_1) + \Delta_0(\cl{F},C_2)
- 2 i_0(C_1,C_2) .\]
\end{lem}
\begin{proof} In view of the definition of the polar excess and the additivity of the multiplicity of a curve, the statement is equivalent to say that
$$
\mu_0(C_1\cup C_2)=\mu_0(C_1)+\mu_0(C_2)+2i_0(C_1,C_2)-1.
$$
This is a classical property of Milnor number of curves (see \cite{casas}, Prop. 6.4.4).
\end{proof}

\begin{prop}
\label{gsv}
Let $\cl{F}$ be a non-dicritical foliation in $(\co^2,0)$ and
$C\subset S_{\mathcal F}$ be a curve union of convergent separatrices of $\cl{F}$. Then
\[ GSV_0(\cl{F},C) = \Delta_0({\cl{F}},C).\]
\end{prop}
\begin{proof} It is enough to consider the case that $C$ is a single separatrix. The general case follows from Lemma \ref{lem:unionexcesos} and the similar statement for the GSV-indices.

Take a reduced equation $f=0$ for $C$ and $\gamma = \gamma(t)$ a Puiseux parametrization.
For a decomposition $ g \omega = k df + f \eta$ such as in \eqref{decomp}, we have
$$
GSV_0(\cl{F},C)  = \frac{1}{2 \pi i} \int_{\partial C} \frac{g}{k} d \left( \frac{k}{g} \right)
   = \frac{1}{2 \pi i} \int_{\partial \mathbb{D}_{\epsilon}}
  \gamma^* \left(\frac{g}{k} d  \left(\frac{ k}{g} \right) \right)
   =  {\rm ord}_{t} {\big (} (k / g) \circ \gamma {\big )},
$$
where $\partial \mathbb{D}_{\epsilon}$ is a small circle around $0 \in \co$.
But, if $(a:b) \in \pcn{1}$, we get from \eqref{decomp} that
\[ a P + b Q = \left(\frac{k}{g} \right) ( a f_x + b f_y) + \frac{f}{g} h  \]
for some $h \in{\mathbb C}\{x,y\}$.
This gives
\[ {\rm ord}_{t} {\big (} (a P + b Q)\circ \gamma {\big )}=
 {\rm ord}_{t} {\big (} (k / g) \circ \gamma {\big )} +  {\rm ord}_{t}  {\big (} ( a f_x + b f_y)\circ \gamma {\big )} .\]
Hence
\[ GSV_0(\cl{F},C)  =
{\rm ord}_{t} {\big (} (a P + b Q)\circ \gamma {\big )} - {\rm
ord}_{t}  {\big (} ( a f_x + b f_y)\circ \gamma {\big )} =
\Delta_0(\cl{F},C).\]
\end{proof}

\begin{obs} In view of  this interpretation of the $GSV$-index in terms of polar intersection numbers, Corollary \ref{cor:generalizada} says that a non-dicritical foliation
 $\cl{F}$  in $(\co^2,0)$ is a generalized curve if and only if $S_{\mathcal F}$ is  convergent and
 $GSV_0(\cl{F},S_{\mathcal F}) = 0$. This
characterization of non-dicritical generalized curves was already
known: its necessity has been proved in \cite{brunella1997II},
whereas its  sufficiency in \cite{lehmann2001}.
\end{obs}

\begin{obs} The $GSV$-index of a foliation $\mathcal F$ in $({\mathbb C}^2,0)$ with respect to a (convergent) curve $C \subset S_{\mathcal F}$ is also equal to
$$
GSV_0({\mathcal F},C)=\dim_{\mathbb C} \frac{{\mathbb C}\{x,y\}}{(f,P,Q)} - \dim_{\mathbb C} \frac{{\mathbb C}\{x,y\}}{(f,f_x,f_y)}
$$
(see \cite{gomezmont1998,licanic}).
Thus, from Proposition~\ref{gsv} and Definition~\ref{def:exceso}, we get that
$$
p_0({\mathcal F},C)  -
\dim_{\mathbb C} \frac{{\mathbb C}\{x,y\}}{(f,P,Q)} = \mu_0(C) - \tau_0(C) +\nu_0(C) -1
$$
where  $\tau_0(C)=\dim_{\mathbb C} \frac{{\mathbb C}\{x,y\}}{(f,f_x,f_y)}$  is the Tjurina number of $C$.
If $S_{\mathcal F}$ is convergent and we define the {\em Tjurina number} $\tau_0({\mathcal F})$ of the foliation $\mathcal F$ as
$$
\tau_0({\mathcal F})=\dim_{\mathbb C} \frac{{\mathbb C}\{x,y\}}{(f,P,Q)}
$$
with $f=0$ an equation of $S_{\mathcal F}$, we obtain that
$$
GSV_0({\mathcal F},S_{\mathcal F})=\tau_0({\mathcal F}) - \tau_0(S_{\mathcal F}).
$$
Hence, we have that  $\mathcal F$ is a generalized curve if and only if $\tau_0({\mathcal F})=\tau_0(S_{\mathcal F})$.
\end{obs}

\section{The Poincar\'e problem and polar multiplicities}
The notion of polar curve of a foliation has a global counterpart.
Let  $\cl{F}$ be a foliation on $\pcn{2}$ with  singular set $\sing(\cl{F})$ and
degree $\deg(\cl{F})$.
Given $q \in \pcn{2}$, the {\em polar curve} of $\cl{F}$ with center $q$ is the closure of
the set of tangencies of $\cl{F}$ with the lines containing $q$. That is:
\[ P_{q}^{\cl{F}}  =  \overline{ \{ p \in \pcn{2} \ba \sing(\cl F);\
 q \in T_p^\pe
  \cl{F}  \} },\]
where $T_{p}^{\pe}{ \cl F}$ is the line through $p$ with direction
$T_{p}\cl{F}$. This is a curve of degree $\deg(\cl{F}) + 1$, except
for the unique degenerate case where $\cl{F}$ is the radial foliation
centered at $q$ (see \cite{mol2010}). This curve contains all points
in $\sing(\cl{F})$, as well as the center $q$. As long as $q$ varies
through $\pcn{2}$, the curves $P_{q}^{\cl{F}}$ form a two-dimensional
linear system, the {\em polar net} of $\cl{F}$. Its base locus is
precisely $\sing(\cl{F})$.

A foliation $\cl{F}$ on $\pcn{2}$ is induced in affine coordinates $\co^2$ by a polynomial
1-form $\omega = P(x,y) dx + Q(x,y) dy$.
The curves with equation $aP(x,y) + b Q(x,y) = 0$, where $(a:b) \in \pcn{1}$, are polar curves
$P_{q}^{\cl{F}}$ with $q \in L_\infty=\pcn{1}$, where $L_\infty$ denotes the line at infinity with respect
 to these affine coordinates. Thus the local polar curves
 can be seen as germs of global polar curves.

 In this Section, we revisit the known proofs of the degree bound for Poincar\'{e} problem in terms of polar curves.
 Namely, we will show that if an algebraic curve $S$ is invariant by a foliation $\cl{F}$ on $\pcn{2}$ then $\deg(S) \leq \deg(\cl{F}) + 2$ in the following cases:
 \begin{enumerate}
 \item[(a)] The singularities of $\cl{F}$ over $S$  are non-dicritical (Carnicer's paper \cite{carnicer1994}).
\item[(b)] The curve $S$ has at most
nodal singularities (Cerveau-Lins Neto's paper \cite{cerveau1991}).
\end{enumerate}

Let $S\subset {\mathbb P}_{\mathbb C}^2$ be a projective curve of degree $d$ invariant by a singular foliation $\cl{F}$ of ${\mathbb P}_{\mathbb C}^2$.
Choose  a line at infinity $L_\infty$ transversal to $S$ that avoids $\sing(\cl{F})$ and $\sing(S)$. Consider affine coordinates in ${\mathbb C}^2={\mathbb P}_{\mathbb C}^2\setminus L_{\infty}$ and fix an affine reduced polynomial equation $f(x,y) = 0$ for $S\setminus L_{\infty}$. Let $\mathcal G$ be the foliation on ${\mathbb P}_{\mathbb C}^2$ defined by $df=0$ in the affine part.
The foliation $\mathcal G$ is the {\em foliation of reference} in \cite{carnicer1994}.

Notice that $\deg(\cl{G}) = d -1 $ and  $L_{\infty}$ is $\cl{G}$-invariant.
The only singularities of
${\mathcal G}$ on $L_{\infty}$ are the $d$ points in $S \cap
L_\infty$. In all such points ${\mathcal G}$ is analytically equivalent to a radial foliation. To see this, it is enough to observe that $\mathcal G$ has the rational first integral
$
F(X,Y,Z)/Z^d
$, where $X,Y,Z$ are homogeneous coordinates and $F$ is the homogeneous polynomial of degree $d$ such that $F(x,y,1)=f(x,y)$.

 Given $q \in \pcn{2}$, let us denote  $\Gamma_q=P^{\mathcal F}_q$ and $\Sigma_q=P^{{\mathcal G}}_q$. There is a nonempty Zariski open set $U\subset {\mathbb P}_{\mathbb C}^2$ such that the following properties hold for any $q\in U$:
 \begin{enumerate}
 \item For any $r\in S\cap L_{\infty}$ we have that $r\notin \Gamma_q$ and $\Sigma_q$ is non singular and transversal to $S$ and $L_{\infty}$ in $r$. In particular $i_r(\Sigma_q,S)=1$ and $i_r(\Gamma_q,S)=0$.
 \item For any $r\in S\setminus (L_{\infty}\cup \sing(\cl{F}))$ we have that either $r\notin \Gamma_q\cup \Sigma_q$ or $r\in \Gamma_q\cap \Sigma_q$ and then $i_r(\Gamma_q,S)=i_r(\Sigma_q,S)=1$.
 \item For any $r\in S\setminus (L_{\infty}\cup \text{Sing}(S))$ we have that either $i_r(\Sigma_q,S)=0$ or $$i_r(\Gamma_q,S)\geq i_r(\Sigma_q,S)=1.$$ In the case $i_r(\Gamma_q,S)=i_r(\Sigma_q,S)$ then $\mathcal F$ is non-singular at $r$.
 \item For any $r\in S\cap\sing(\cl{F})$ we have that $\Gamma_q$, respectively $\Sigma_q$, is a $S$-polar generic curve for $\mathcal F$, respectively for $\mathcal G$. In particular, when $\mathcal F$ is non-dicritical at $r$ we have
     $$
     \Delta_r({\mathcal F},S)=i_r(\Gamma_q,S)-i_r(\Sigma_q,S)\geq 0,
     $$
in view of Corollary \ref{cor:positivity}. If $r$ is a nodal singularity of $S$, we have $i_r(\Sigma_q,S)=2$  (note that $S$ has two branches at $r$) and thus  $i_r(\Gamma_q,S)-i_r(\Sigma_q,S)\geq 0$.
 \end{enumerate}
For the second statement, let us note that up to choosing a generic $q$, the line joining $q$ and $r$, for $r\in S\setminus (L_{\infty}\cup \sing(\cl{F}))$ is either transversal to $S$ of tangent with intersection multiplicity equal to $2$. In the first case we have that $r\notin \Gamma_q\cup \Sigma_q$. Consider the second case. Let us choose affine coordinates $(x,y)$ centered at $r$ such that the lines passing through $q$ are the leaves of $dx=0$. Moreover, the curve $S$ is locally given at $r$ as $y^2+xg(x,y)=0$ where $g(x,y)\in {\mathbb C}\{x,y\}$ is such that $g(0,0)\ne 0$. Since $\mathcal F$ is non singular at $r$, it is locally given at $r$ by $\eta=0$ where
$$
\eta= d(U(x,y)(y^2+xg(x,y))), \quad U(0,0)\ne 0.
$$
Hence $\Gamma_q$ is given at $r$ by $h(x,y)=0$, for $dx\wedge\eta=h(x,y)dx\wedge dy$. That is
$$
h(x,y)= \frac{\partial (U(x,y)(y^2+xg(x,y)))}{\partial y}= 2yU(x,y)+y^2\frac{\partial U(x,y)}{\partial y}+x(\cdots).
$$
This implies that $i_r(\Gamma_q,S)=1$. Same argument to see that $i_r(\Sigma_q,S)=1$. These kind of computations also show the third statement.

Now, let us consider a generic $q$ as above. By B\'{e}zout's theorem applied to the curves $S$ and $\Sigma_q$, we have
$$
d^2=\sum_{r\in {\mathcal A}}i_r(\Sigma_q,S)+ \sum_{r\in {\mathcal B}}i_r(\Sigma_q,S)+ \sum_{r\in {\mathcal C}}i_r(\Sigma_q,S),
$$
where ${\mathcal A}=S\cap L_\infty $, ${\mathcal B}=S\setminus (L_\infty\cup\sing(\cl{F}))$ and
${\mathcal C}=S\cap\sing(\cl{F})$. We obtain that
$$
d(d-1)=\sum_{r\in {\mathcal B}}i_r(\Sigma_q,S)+ \sum_{r\in {\mathcal C}}i_r(\Sigma_q,S).
$$
By B\'{e}zout's theorem applied to the curves $S$ and $\Gamma_q$, we obtain
$$
d(\deg({\mathcal F})+1)=\sum_{r\in {\mathcal B}}i_r(\Gamma_q,S)+ \sum_{r\in {\mathcal C}}i_r(\Gamma_q,S).
$$
Taking the difference, we have
\begin{equation}
\label{eq:cotaspoincare}
\deg({\mathcal F})+2-d=\frac{1}{d}\sum_{r\in S\cap \text{Sing}({\mathcal F})}(i_r(\Gamma_q,S)-i_r(\Sigma_q,S))\geq 0.
\end{equation}
Note that in the non-dicritical case this equation reads as
\begin{equation}
\label{eq:cotaspoincare2}
\deg({\mathcal F})+2-d=\frac{1}{d}\sum_{r\in S\cap \text{Sing}({\mathcal F})}\Delta_r({\mathcal F},S).
\end{equation}

\begin{obs}
By Proposition \ref{gsv}, the number in the right part of Equation \eqref{eq:cotaspoincare}
is $(1/d)\sum_{p \in {\rm Sing}(\cl{F}) \cap S} GSV_p(\cl{F},S)$,
which is obtained  in \cite{brunella1997I} and \cite{brunella1997II} as
$c_1(N_\cl{F}) \cdot S - S \cdot S$, where $N_\cl{F} = \cl{O}(d + 2)$ is the normal bundle of $\cl{F}$.
\end{obs}

\subsection{Logarithmic foliations} The limit case for the bound of degrees is closely related with logarithmic forms, as shown in the following statement
\begin{teo}[\cite{cerveau1991}, \cite{brunella1997II}]
 \label{teo3}
 Take a curve $S\subset {\mathbb P}^2_{\mathbb C}$  given by a homogeneous polynomial equation
$
P=P_1P_2\cdots P_n=0
$,
 where each polynomial $P_i$ is irreducible of degree $d_i$. Suppose that $S$ is invariant for a foliation $\mathcal F$ of ${\mathbb P}^2_{\mathbb C}$ that is non-dicritical at each point  $q\in S$. The following statements are equivalent
\begin{enumerate}
\item $\deg (S)=\deg({\mathcal F})+2$.

\item There are residues $\lambda_i\in {\mathbb C}^*$ with $\sum_{i=1}^n\lambda_id_i=0$ such that $\mathcal F$ is given by $W=0$, where $W$ is the global closed logarithmic $1$-form in ${\mathbb P}^2_{\mathbb C}$ defined by
    $$
    W=\sum_{i=1}^n\lambda_i \frac{dP_i}{P_i}.
    $$
\item The foliation ${\mathcal F}$ is a generalized curve at any  $q\in S$ and $S$ contains all the separatrices of $\mathcal F$ at $q$.
\end{enumerate}
\end{teo}

 The case of nodal singularities has been proved by D. Cerveau and A. Lins Neto  in \cite{cerveau1991} and it is also valid in a dicritical situation. Later, the result was extended by M. Brunella to the
non-dicritical case in (\cite{brunella1997II}, Proposition 10). His proof relies on computations involving indices of vector fields for a subsequent application of the following result of Deligne:
\begin{teorm}
[Deligne  \cite{deligne1971}]
 \label{teo:deligne}
 Let $\omega$ be a logarithmic 1-form on a projective variety $M$ whose polar   divisor has
normal crossings. Then $\omega$ is closed (\em and hence it defines a codimension one logarithmic foliation on $M$).
\end{teorm}
Recently
D. Cerveau gave a very short and elegant proof of Theorem \ref{teo3} in \cite{cerveau2013}. This approach is based on the fact that a non-dicritical logarithmic meromorphic $1$-form $\omega$ (that is $\omega$ and $d\omega$ have at most simple poles) has also this property after a blow-up.

Concerning polar invariants, we can see as a consequence of Corollary \ref{cor:generalizada} and Equation~\eqref{eq:cotaspoincare2}  that the statements (1) and (3) of Theorem \ref{teo3} are equivalent.

Let us end these notes by providing a proof  that (1) implies (2) following the arguments in \cite{cerveau2013}, \cite{brunella1997II} and \cite{Bru-M}  of D. Cerveau, M. Brunella and L. G. Mendes, but avoiding the direct use of Deligne's statement.

Let us consider a homogeneous 1-form $$
T=\sum_{i=0}^2A_i(X_0,X_1,X_2)dX_i;\quad  \sum_{i=0}^2X_iA_i=0
$$
defining $\mathcal F$ (see \cite{Cano-Cerveau}), where $\deg{A_i}=\deg{\mathcal F}+1$ and consider the meromorphic 1-form $\Omega=T/P$. Since $\deg(P)=\deg{\mathcal F}+2$, we see that $\Omega$ defines a global meromorphic 1-form $\omega$ on ${\mathbb P}^2_{\mathbb C}$. Looking at each point $p\in S$, by Proposition 2.1 in \cite{cerveau2013}, the form $\omega$ is logarithmic at $p$. We apply the stability by non-dicritical blowing-up (Proposition 2.2 in \cite{cerveau2013}) of being logarithmic, to see that after a reduction of singularities
$$
\pi:\widetilde{\mathbb P}^2_{\mathbb C}\rightarrow {\mathbb P}^2_{\mathbb C}
$$
of $\mathcal F$ along $S$, we obtain a logarithmic 1-form $\pi^*\omega$ that has locally one of the following expressions  at a point $p$ in the total transform $\tilde{S}$ of $S$:
$$
U(x,y) (\lambda_p \frac{dx}{x} + b(x,y)dy); \quad  U(x,y)(\lambda_p \frac{dx}{x}+b(x,y) \frac{dy}{y})
$$
with $U(0,0)=1$, where $x=0$ is a selected irreducible component of the total transform $\tilde S$ of $S$. The first case corresponds to a non-singular point in $\tilde S$ and in the second one we have $S=(xy=0)$ (note that we can do the same argument for the component $y=0$). The functions $\lambda_p$ are holomorphic, hence they are constant. So we can attach a residue
$\lambda_i$ to each irreducible component $P_i=0$ of $S$. Moreover, taking a general line $\ell$ in ${\mathbb P}^2_{\mathbb C}$ avoiding the singular locus of $P$, the sum of residues  in $\omega\vert_\ell$ gives
$$
\sum_{i=1}^n\lambda_i\deg(P_1)=0.
$$
Now, we consider the global meromorphic 1-form
$$
    W=\sum_{i=1}^n\lambda_i \frac{dP_i}{P_i}.
$$
It follows that $\Omega-W$ is holomorphic (the residues coincide) and hence $\Omega=cW$ for a non-null constant $c\in {\mathbb C}$.

%
%


\medskip \medskip  \medskip
\noindent
Felipe Cano \\
Departamento de
\'Algebra, An\'alisis Matem\'atico,
Geometr\'\i a y Topolog\'\i a \\
Universidad de Valladolid\\
Paseo de Bel\'en 7,
47011 -- Valladolid, SPAIN  \\
fcano@agt.uva.es

\medskip \medskip
\noindent
Nuria Corral \\
Departamento de Matem\'aticas, Estad\'\i stica y Computaci\'on \\
Universidad de Cantabria \\
Avda. de los Castros s/n, 39005 -- Santander, SPAIN \\
nuria.corral@unican.es
  \medskip \medskip

\noindent
Rog\'erio  Mol  \\
Departamento de Matem\'atica \\
Universidade Federal de Minas Gerais \\
Av. Ant\^onio Carlos, 6627  \  C.P. 702  \\
30123-970  --
Belo Horizonte -- MG,
BRAZIL \\
rsmol@mat.ufmg.br

\end{document}